\documentclass[12pt, reqno]{amsart}
\makeatletter
\@namedef{subjclassname@1991}{$\mathrm{1991}$ Mathematics Subject Classification}
\@namedef{subjclassname@2000}{$\mathrm{2000}$ Mathematics Subject Classification}
\@namedef{subjclassname@2010}{$\mathrm{2010}$ Mathematics Subject Classification}
\@namedef{subjclassname@2020}{$\mathrm{2020}$ Mathematics Subject Classification}
\makeatother
\usepackage{amsmath,amsthm, amscd, amsfonts, amssymb, graphicx, color}
\usepackage[bookmarksnumbered, colorlinks, plainpages,linkcolor=blue,urlcolor=blue,citecolor=blue]{hyperref}
\newtheorem{thm}{Theorem}[section]
\newtheorem{cor}[thm]{Corollary}
\newtheorem{lem}[thm]{Lemma}

\newtheorem{defn}[thm]{Definition}

\newtheorem{exam}[thm]{Example}
\numberwithin{equation}{section}

\begin{document}

\title{The higher-order group inverse in a ring}

\author{Dayong Liu}
\address{College of Computer and Mathematics, Central South University of Forestry and Technology, Changsha 410004, China}
\email{<liudy@csuft.edu.cn>}

\author{Huanyin Chen}
\address{School of Big Data, Fuzhou University of International Studies and Trade, Fuzhou 350202, China}
\email{<huanyinchenfz@163.com>}

\subjclass[2020]{16U99, 15A09, 46H05.} \keywords{group inverse; higher-order group inverse;
weak higher-order group inverse; nilpotent, ring.}

\begin{abstract} This paper introduces and studies the higher-order group inverse in a ring. We extend known properties of the higher-order group inverse from complex matrices to elements of a ring and, in the process, derive new results. We further characterize when a ring element can be expressed as the sum of a higher-order group element and a nilpotent element.\end{abstract}

\maketitle

\section{Introduction}

A ring $R$ is called a *-ring if it is equipped with an involution $*: R \to R$, denoted by $x \mapsto x^*$, which satisfies the following properties for all $x, y \in R$ and $\lambda \in \mathbb{C}$:
$$(x+y)^*=x^*+y^*, (\lambda x)^*=\overline{\lambda}x^*,(xy)^*=y^*x^*,(x^*)^*=x.$$ In linear algebra, the concept of a matrix inverse is fundamental. However, the conventional inverse exists only for nonsingular square matrices. To handle singular or rectangular matrices, the theory of weak inverses was developed.

Among these, the Moore-Penrose inverse is the most well-known. An element $a \in R$ is said to have a Moore-Penrose inverse if there exists an element $x \in R$ such that
$$xax=x, axa=a,(ax)^*=ax, (xa)^*=xa.$$ This element $x$ is unique if it exists, and we denote it by $a^{\dagger}$. The set of all Moore-Penrose invertible elements in $R$ is denoted by $R^{\dagger}$ (see~\cite{M0,R}). Other weak inverses are tailored for specific structural or algebraic requirements.

An element $a \in R$ has a group inverse if there exists an element $x \in R$ such that
$$xax=x, axa=a, ax=xa.$$ Such an element $x$ is unique if it exists, is denoted by $a^{\#}$, and is called the group inverse of $a$. For a complex matrix, the group inverse exists if and only if the matrix has index one, meaning its algebraic and geometric multiplicities for the zero eigenvalue are equal.

The group inverse has emerged as a powerful tool in linear algebra and its numerous applications. Distinct from the more widely known Moore-Penrose inverse, the group inverse exists for a singular matrix under the specific condition that its index is one. The group inverse has been extensively studied from many points of view, including the exploration of its theoretical characteristics, the development of efficient computational algorithms, and its wide-ranging applications in fields such as Markov chains, singular differential equations, control theory, and graph theory (see~\cite{M}).

A significant challenge is extending the group inverse to higher-order settings in a way that preserves most of its key properties while ensuring the new definition applies to a broader class of ring elements. The higher-order group inverse for complex matrices was introduced and studied by Wang via the Hartwig-Spindelbock decomposition (see~\cite{W1}). Since this decomposition is not available for the broader class of objects, e.g., matrices over a ring, we are motivated to generalize the concept of the higher-order group inverse to this new context. This paper, therefore, investigates the higher-order group inverse in this general setting by developing a new methodology.

In Section 2, we introduce the higher-order group inverse for an ring element and
 characterize it by means of the solvability of equations system in a ring $R$.

 \begin{defn} An element $a\in R^{\dag}$ has the higher-order group inverse ($\mathcal{H}$-group inverse for short) if there exists $x\in aR\bigcap Ra$ such that $$xax=x, a^2xa^2=a^3, (a^2xa^*)^*=a^2xa^*, (a^*xa^2)^*=a^*xa^2.$$ If such $x$ exists, it is unique, and denote it by $a^{\mathcal{H}}$. The set of all $\mathcal{H}$-group invertible elements in $R$ is denoted by $R^{\mathcal{H}}$.\end{defn}

In Section 2, we prove that $a\in R^{\dag}$ has $\mathcal{H}$-group inverse if and only if $a^{\dag}a^3a^{\dag}\in R^{\dag}$. The representation of the $\mathcal{H}$-group inverse of a ring element is presented.

The weak group inverse for square complex matrices was first introduced and studied by Wang and Chen (see~\cite{W}). Subsequently, Zou et al. generalized this concept from complex matrices to elements of a ring with a proper involution (see~\cite{Z}). A key property of this inverse is that it characterizes ring elements that can be expressed as the sum of a group-invertible element and a nilpotent element. Consequently, the weak group inverse has been extensively studied from various perspectives~\cite{F,L,MD2,MD21,MD3,Z}. This inspires us to characterize when a ring element can be written as the sum of a higher-order group-invertible element and a nilpotent.

Let $R^{nil}$ be the set of all nilpotent elements in $R$. We adopt:

\begin{defn} An element $a\in R$ has weak $\mathcal{H}$-group inverse if there exist $x,y\in R$ such that $$a=x+y, x^*y=0, yx=0,
x\in R^{\mathcal{H}}, y\in R^{nil}.$$\end{defn}

We denote $x^{\mathcal{H}}$ by $a^{\tiny\textcircled{\tiny H}}$ and call it the weak $\mathcal{H}$-group inverse of $a$. The set of all weak $\mathcal{H}$-group invertible elements in $R$ is denoted by $R^{\tiny\textcircled{H}}$.

Recall that an element $a\in R$ has weak Moore-Penrose inverse if there exists $x\in R$ such that
$$x=xax, (ax)^*=ax, (xa)^*=xa, a-axa\in R^{nil}.$$ The preceding $x$ is unique if it exists, and we denote it by $a^{\tiny\textcircled{\dag}}$. The set of all generalized Moore-Penrose invertible elements in $R$ is denoted by $R^{\tiny\textcircled{\dag}}$ (see~\cite{C}).

In Section 3, we establish the relations between the weak $\mathcal{H}$-group inverse and weak Moore-Penrose inverse. It is proved that
$a\in R^{\tiny\textcircled{H}}$ if and only if $a\in R^{\tiny\textcircled{\dag}}$ and $a^{\tiny\textcircled{\dag}}a^3a^{\tiny\textcircled{\dag}}\in R^{\dag}$. The element-wise characterization of the weak $\mathcal{H}$-group inverse is thereby presented (see~Theorem 3.4).

Throughout the paper, all *-ring are associative with an identity. $R^{D}, R^{\dag}, R^{\tiny\textcircled{\dag}}$ and $R^{nil}$ denote the sets of all Drazin,
Moore-Penrose, weak Moore-Penrose invertible and nilpotent elements in the ring $R$, respectively. Let $a\in R$. Set $im(a)=\{ ax ~|~ x\in R\}$ and
$ker(a)=\{ x\in R ~|~ ax=0\}.$ We use $p_{im(a)}$ to denote the projection $p$ such that $im(p)=im(a)$.

\section{the $\mathcal{H}$-group inverse}

In this section, we characterize the $\mathcal{H}$-group inverse of a ring element using its Moore-Penrose inverse. Our starting point is as follows:

\begin{thm} Let $a\in R$. Then the following are equivalent:\end{thm}
\begin{enumerate}
\item [(1)] $a\in R^{\mathcal{H}}$.
\vspace{-.5mm}
\item [(2)] $a\in R^{\dag}$ and $a^{\dag}a^3a^{\dag}\in R^{\dag}$.
\end{enumerate}
In this case, $$a^{\mathcal{H}}=[a^{\dag}a^3a^{\dag}]^{\dag}.$$
\begin{proof} $(1)\Rightarrow (2)$ By hypothesis, there exists $x\in aR\bigcap Ra$ such that $$xax=x, a^2xa^2=a^3, (a^2xa^*)^*=a^2xa^*, (a^*xa^2)^*=a^*xa^2.$$
Write $x=ar=sa$ for some $r,s\in R$. We directly verify that
$$\begin{array}{rll}
x(a^{\dag}a^3a^{\dag})&=&ar(a^{\dag}a^3a^{\dag})=(aa^{\dag})ar(a^{\dag}a^3a^{\dag})\\
&=&(a^{\dag})^*a^*x(a^{\dag}a^3a^{\dag})\\
&=&(a^{\dag})^*a^*(saa^{\dag}a)a^2a^{\dag}\\
&=&(a^{\dag})^*a^*(sa)a^2a^{\dag}\\
&=&(a^{\dag})^*[a^*xa^2]a^{\dag},\\
\big(a^{\dag}a^3a^{\dag}\big)x&=&\big(a^{\dag}a^3a^{\dag}\big)ar\\
&=&(a^{\dag}a^2(aa^{\dag}a))r=(a^{\dag}a^2))x\\
&=&(a^{\dag}a^2))saa^{\dag}a=(a^{\dag}a^2))x(a^{\dag}a)^*\\
&=&a^{\dag}(a^2xa^*)(a^{\dag})^*,\\
\big(x(a^{\dag}a^3a^{\dag})\big)^*&=&x(a^{\dag}a^3a^{\dag}),\\
\big((a^{\dag}a^3a^{\dag})x\big)^*&=&(a^{\dag}a^3a^{\dag})x.\end{array}$$ Moreover, we see that
$$\begin{array}{rll}
(a^{\dag}a^3a^{\dag})x(a^{\dag}a^3a^{\dag})&=&(a^{\dag}a^3a^{\dag})ar(a^{\dag}a^3a^{\dag})\\
&=&(a^{\dag}a^2)x(a^{\dag}a^3a^{\dag})=(a^{\dag}a^2)sa(a^{\dag}a^3a^{\dag})\\
&=&(a^{\dag}a^2)x(a^2a^{\dag})=a^{\dag}(a^2xa^2)a^{\dag}\\
&=&a^{\dag}a^3a^{\dag},\\
x(a^{\dag}a^3a^{\dag})x&=&sa(a^{\dag}a^3a^{\dag})ar\\
&=&sa(a^{\dag}a^3a^{\dag})ar=s((aa^{\dag}a)a(aa^{\dag}a))r\\
&=&(sa)a(ar)=xax=x.
\end{array}$$
Therefore $a^{\dag}aa^{\dag}\in R^{\dag}$ and $(a^{\dag}aa^{\dag})^{\dag}=x$. This implies that $x$ is unique, as required.

$(2)\Rightarrow (1)$ Let $x=(a^{\dag}a^3a^{\dag})^{\dag}.$ Then $$\begin{array}{rll}
x&=&(a^{\dag}a^3a^{\dag})^{\dag}=(a^{\dag}a^3a^{\dag})^{\dag}(a^{\dag}a^3a^{\dag})(a^{\dag}a^3a^{\dag})^{\dag}\\
&=&(a^{\dag}a^3a^{\dag})^{\dag}(a^{\dag}a^3a^{\dag})(aa^{\dag})(a^{\dag}a^3a^{\dag})^{\dag}\\
&=&(aa^{\dag})(a^{\dag}a^3a^{\dag})^{\dag}\\
&\in&aR.
\end{array}$$ Likewise, we have $x\in Ra$. We further check that
$$\begin{array}{rll}
xax&=&[a^{\dag}a^3a^{\dag}]^{\dag}a[a^{\dag}a^3a^{\dag}]^{\dag}\\
&=&[a^{\dag}a^3a^{\dag}]^{\dag}(a^{\dag}a)a(aa^{\dag})[a^{\dag}a^3a^{\dag}]^{\dag}\\
&=&[a^{\dag}a^3a^{\dag}]^{\dag}[a^{\dag}a^3a^{\dag}][a^{\dag}a^3a^{\dag}]^{\dag}=x,\\
a^2xa^2&=&a^2(a^{\dag}a^3a^{\dag})^{\dag}a^2=a^3a^{\dag}[a^{\dag}a^3a^{\dag}]^{\dag}a^{\dag}a^3\\
&=&[aa^{\dag}a^3]a^{\dag}[a^{\dag}a^3a^{\dag}]^{\dag}a^{\dag}[a^3a^{\dag}a]\\
&=&a[a^{\dag}a^3a^{\dag}][a^{\dag}a^3a^{\dag}]^{\dag}[a^{\dag}a^3a^{\dag}]a\\
&=&a[a^{\dag}a^3a^{\dag}]a=a^3,\\
a^2xa^*&=&a^2[a^{\dag}a^3a^{\dag}]^{\dag}a^*=a^3a^{\dag}[a^{\dag}a^3a^{\dag}]^{\dag}a^*\\
&=&a[a^{\dag}a^3a^{\dag}][a^{\dag}a^3a^{\dag}]^{\dag}a^*\\
a^*xa^2&=&a^*[a^{\dag}a^3a^{\dag}]^{\dag}a^2=a^*[a^{\dag}a^3a^{\dag}]^{\dag}a^{\dag}a^3\\
&=&a^*[a^{\dag}a^3a^{\dag}]^{\dag}[a^{\dag}a^3a^{\dag}]a,\\
(a^2xa^*)^*&=&a^2xa^*,\\
(a^*xa^2)^*&=&a^*xa^2.
\end{array}$$
Therefore $x\in aR\bigcap Ra$ is the solution of the system of equations: $$xax=x, a^2xa^2=a^3, (a^2xa^*)^*=a^2xa^*, (a^*xa^2)^*=a^*xa^2,$$ as required.
\end{proof}

\begin{cor} Let $a\in R^{\dag}\bigcap R^{\#}$. Then
$a\in R^{\mathcal{H}}$ and $a^{\mathcal{H}}=a^{\#}$.\end{cor}
\begin{proof} We verify that $$\begin{array}{rll}
a^{\#}[a^{\dag}a^3a^{\dag}]&=&[(a^{\#})^2a][a^{\dag}a^3a^{\dag}]\\
&=&(a^{\#})^2(aa^{\dag}a)a^2a^{\dag}\\
&=&(a^{\#})^2a^3a^{\dag}=aa^{\dag}.
\end{array}$$ Likewise, we have
$[a^{\dag}a^3a^{\dag}]a^{\#}=a^{\dag}a$. Hence
$$[a^{\#}(a^{\dag}a^3a^{\dag})]^*=a^{\#}(a^{\dag}a^3a^{\dag}), [(a^{\dag}a^3a^{\dag})a^{\#}]^*=(a^{\dag}a^3a^{\dag})a^{\#}.$$
Moreover, we check that $$\begin{array}{rll}
[a^{\dag}a^3a^{\dag}]a^{\#}[a^{\dag}a^3a^{\dag}]&=&[a^{\dag}a^3a^{\dag}][aa^{\dag}]=a^{\dag}a^3a^{\dag},\\
(a^3a^{\dag})a^{\#}(a^3a^{\dag})&=&[aa^{\dag}](a^3a^{\dag})=a^3a^{\dag}.
\end{array}$$
By virtue of Theorem 2.1, $a^{\mathcal{H}}=[a^{\dag}a^3a^{\dag}]^{\dag}=a^{\#}$, as asserted.\end{proof}

We come now to establish the representation of the $\mathcal{H}$-group inverse by using certain projections.

\begin{thm} Let $a\in R^{\mathcal{H}}$. The system given by
$$ax=a[q_aap_a]^{\dag}, im(x)\subseteq im(p_aa^*q_a)$$ is consistent and its unique solution is $x=a^{\mathcal{H}}.$
\end{thm}
\begin{proof} Set $x=a^{\mathcal{H}}$. Then $ax=a[a^{\dag}a^3a^{\dag}]^{\dag}=a[q_aap_a]^{\dag}$ and $x=[a^{\dag}a^3a^{\dag}]^{\dag}=[a^{\dag}a^3a^{\dag}]^{\dag}[a^{\dag}a^3a^{\dag}][a^{\dag}a^3a^{\dag}]^{\dag}\in [q_aap_a]^*R$; hence, $im(x)\subseteq  [q_aap_a]^*R$.

Suppose that $$ax=a[q_aap_a]^{\dag}, im(x)\subseteq im(p_aa^*q_a).$$ Since $p_aa^*q_a=(q_aap_a)^*=[(q_aap_a)(q_aap_a)^{\dag}(q_aap_a)]^*=(q_aap_a)^{\dag}(q_aap_a)(q_aap_a)^*$. Write $x=(q_aap_a)^{\dag}z$ for a $z\in R$.
Then we verify that $$\begin{array}{rll}
x&=&(q_aap_a)^{\dag}z\\
&=&(q_aap_a)^{\dag}(q_aap_a)(q_aap_a)^{\dag}z\\
&=&(q_aap_a)^{\dag}q_a(ax)\\
&=&(q_aap_a)^{\dag}q_aa[q_aap_a]^{\dag}\\
&=&(q_aap_a)^{\dag}[q_aap_a][q_aap_a]^{\dag}\\
&=&(q_aap_a)^{\dag},
\end{array}$$ as asserted.\end{proof}

\begin{cor} Let $a\in R^{\mathcal{H}}$. The system given by
$$p_ax=[q_aap_a]^{\dag}, im(x)\subseteq im(a)$$ is consistent and its unique solution is $x=a^{\mathcal{H}}.$
\end{cor}
\begin{proof} One easily checks that $$p_aa^{\mathcal{H}}=
aa^{\dag}[a^{\dag}a^3a^{\dag}]^{\dag}=[a^{\dag}a^3a^{\dag}]^{\dag}=(q_aap_a)^{\dag}.$$ In view of Theorem 2.3,
$$im(a^{\mathcal{H}})\subseteq im(p_aa^*q_a)\subseteq im(a).$$

Suppose that $p_ax_i=(q_aap_a)^{\dag}, im(x_i)\subseteq im(a)$ for $i=1,2$. Then
$p_ax_1=p_ax_2$; hence, $x_1-x_2\in ker(p_a)\bigcap im(a)=0$. Therefore $x_1=x_2$, as asserted.\end{proof}

Let $a,b,c\in R$. An element $a$ has $(b,c)$-inverse provide that there exists $x\in R$ such that $$xab=b, cax=c ~\mbox{and}~ x\in bRx\bigcap xRc.$$ If such $x$ exists, it is unique and denote it by $a^{(b,c)}$ (see~\cite{D}).

\begin{thm} Let $a\in R^{\mathcal{H}}$. Then
$$a^{\mathcal{H}}=a^{\big(p_a(a^*)^2, (a^*)^2q_a\big)}.$$\end{thm}
\begin{proof} Obviously, we have $a\in R^{\dag}$. Let $x=a^{\mathcal{H}}$. We check that
$$\begin{array}{rll}
x&=&[a^{\dag}a^3a^{\dag}]^{\dag}=[a^{\dag}a^3a^{\dag}]^{\dag}[a^{\dag}a^3a^{\dag}][a^{\dag}a^3a^{\dag}]^{\dag}\\
&=&[a^{\dag}a^3a^{\dag}]^*\big([a^{\dag}a^3a^{\dag}]^{\dag}\big)^*x\in p_a(a^*)^2Rx,\\
x&=&x\big([a^{\dag}a^3a^{\dag}][a^{\dag}a^3a^{\dag}]^{\dag}\big)^*=x\big([a^{\dag}a^3a^{\dag}]^{\dag}\big)^*[a^{\dag}a^3a^{\dag}]^*]\\
&\in& xR(a^*)^2q_a,\\
xa[aa^{\dag}(a^*)^2]&=&[a^{\dag}a^3a^{\dag}]^{\dag}a[aa^{\dag}(a^*)^2]\\
&=&[a^{\dag}a^3a^{\dag}]^{\dag}a(a^3a^{\dag})^*\\
&=&[a^{\dag}a^3a^{\dag}]^{\dag}a[a^{\dag}a^3a^{\dag}]^*a^*\\
&=&[a^{\dag}a^3a^{\dag}]^{\dag}[a^{\dag}a]a[aa^{\dag}][a^{\dag}a^3a^{\dag}]^*a^*\\
&=&[a^{\dag}a^3a^{\dag}]^{\dag}[a^{\dag}a^3a^{\dag}][a^{\dag}a^3a^{\dag}]^*a^*\\
&=&\big([a^{\dag}a^3a^{\dag}]a^{\dag}a^3a^{\dag}]^{\dag}[a^{\dag}a^3a^{\dag}]\big)^*a^*\\
&=&\big(a^{\dag}a^3a^{\dag}\big)^*a^*\\
&=&(a^3a^{\dag})^*=aa^{\dag}(a^*)^2,\\
(a^*)^2a^{\dag}a^2x&=&(a^*)^2a^{\dag}a^2[a^{\dag}a^3a^{\dag}]^{\dag}\\
&=&(a^*)^2a^{\dag}a^2[aa^{\dag}][a^{\dag}a^3a^{\dag}]^{\dag}\\
&=&(a^*)^2[a^{\dag}a][a^{\dag}a^3a^{\dag}][a^{\dag}a^3a^{\dag}]^{\dag}\\
&=&[a^{\dag}a^3]^*\big([a^{\dag}a^3a^{\dag}][a^{\dag}a^3a^{\dag}]^{\dag}\big)^*\\
&=&\big([a^{\dag}a^3a^{\dag}][a^{\dag}a^3a^{\dag}]^{\dag}[a^{\dag}a^3]\big)^*\\
&=&\big([a^{\dag}a^3a^{\dag}][a^{\dag}a^3a^{\dag}]^{\dag}[a^{\dag}a^3a^{\dag}]a\big)^*\\
&=&\big([a^{\dag}a^3a^{\dag}]a\big)^*\\
&=&\big(a^{\dag}a^3\big)^*=(a^*)^2a^{\dag}a.
\end{array}$$ Therefore $$a^{\mathcal{H}}=a^{\big(p_a(a^*)^2, (a^*)^2q_a\big)},$$ as asserted.
\end{proof}

Let $a\in R$. We say that $a$ has $\{2\}$-inverse $x$ provided that $x=xax$. We denote $a^{(2)}_{T,S}=\{ x\in R~|~xax=x, im(a)=T, ker(a)=S\}$. We next consider the relation between the $\mathcal{H}$-group inverse and $\{2\}$-inverse in a ring.

\begin{thm} Let $a\in R^{\mathcal{H}}$. Then $$a^{\mathcal{H}}=a^{(2)}_{im\big(p_a(a^*)^2\big), ker\big((a^*)^2q_a\big)}.$$\end{thm}
\begin{proof} Let $x=a^{\mathcal{H}}$. Then we have $x=xax$.

Step 1. $im(x)=im\big(p_a(a^*)^2\big)$. In view of Theorem 2.1, we have

$$\begin{array}{rll}
x&=&[a^{\dag}a^3a^{\dag}]^{\dag}\\
&=&[a^{\dag}a^3a^{\dag}]^{\dag}[a^{\dag}a^3a^{\dag}][a^{\dag}a^3a^{\dag}]^{\dag}\\
&=&\big([a^{\dag}a^3a^{\dag}]^{\dag}[a^{\dag}a^3a^{\dag}]]\big)^*[a^{\dag}a^3a^{\dag}]^{\dag}\\
&=&[aa^{\dag}(a^*)^2]\big([a^{\dag}a^3a^{\dag}]^{\dag}a^{\dag}\big)^*[a^{\dag}a^3a^{\dag}]^{\dag}\\
&\in&im\big(aa^{\dag}(a^*)^2\big),
\end{array}$$
$$\begin{array}{rll}
[aa^{\dag}(a^*)^2]^*&=&a^3a^{\dag}=a[a^{\dag}a^3a^{\dag}]x[a^{\dag}a^3a^{\dag}]\\
&=&a[a^{\dag}a^3a^{\dag}][a^{\dag}a^3a^{\dag}]^*x^*.
\end{array}$$
Hence, $aa^{\dag}(a^*)^2=x[a^{\dag}a^3a^{\dag}]\big(a[a^{\dag}a^3a^{\dag}]\big)^*\in im(x)$.
Therefore $im(x)=im\big(p_a(a^*)^2\big)$.

Step 2. $ker(x)=ker\big((a^*)^2q_a\big)$. If $(a^*)^2q_ar=0$ for some $r\in R$, then $$\begin{array}{rll}
xr&=&[q_aap_a]^{\dag}r\\
&=&[q_aap_a]^{\dag}[q_aap_a][q_aap_a]^{\dag}r\\
&=&[q_aap_a]^{\dag}\big([q_aap_a][q_aap_a]^{\dag}\big)^*r\\
&=&[q_aap_a]^{\dag}\big(p_a[q_aap_a]^{\dag}\big)^*[(a^2)^*q_ar]\\
&=&0.
\end{array}$$ Hence, $r\in ker(x)$.

If $x(r)=0$, then $$\begin{array}{rll}
(a^2)^*q_ar&=&[a^{\dag}a^3]^*r=[(a^{\dag}a^3a^{\dag})a]^*r=a^*[a^{\dag}a^3a^{\dag}]^*r\\
&=&a^*[(a^{\dag}a^3a^{\dag})(a^{\dag}a^3a^{\dag})^{\dag}(a^{\dag}a^3a^{\dag})]^*r\\
&=&a^*[(a^{\dag}a^3a^{\dag})]^*[(a^{\dag}a^3a^{\dag})(a^{\dag}a^3a^{\dag})^{\dag}]^*r\\
&=&a^*[(a^{\dag}a^3a^{\dag})]^*[a^{\dag}a^3a^{\dag}](xr)=0.
\end{array}$$ Thus, $r\in ker\big((a^2)^*q_a\big)$. Hence $ker(x)=ker\big((a^*)^2q_a\big)$.

Therefore we complete the proof.\end{proof}

\section{weak $\mathcal{H}$-group inverse}

The purpose of this section is to determine under what conditions a ring element can be decomposed into the sum of an
$\mathcal{H}$-group invertible element and a nilpotent element. To this end, we first establish the connection between the weak higher-order group inverse and the weak Moore-Penrose inverse.

\begin{thm} Let $a\in R$. Then the following are equivalent:\end{thm}
\begin{enumerate}
\item [(1)] $a\in R^{\tiny\textcircled{H}}$.
\vspace{-.5mm}
\item [(2)] $a\in R^{\tiny\textcircled{\dag}}$ and $a^{\tiny\textcircled{\dag}}a^3a^{\tiny\textcircled{\dag}}\in R^{\dag}$.
\end{enumerate}
In this case, $$a^{\tiny\textcircled{H}}=[a^{\tiny\textcircled{\dag}}a^3a^{\tiny\textcircled{\dag}}]^{\dag}.$$
\begin{proof} $(1)\Rightarrow (2)$ Since $a\in R^{\tiny\textcircled{H}}$, there exist $x,y\in R$ such that $$a=x+y, x^*y=xy^*=0, x\in R^{\mathcal{H}}, y\in R^{nil}.$$ As $x\in R^{\mathcal{H}}$, we have $x\in R^{\dag}$. This implies that
$a\in R^{\tiny\textcircled{\dag}}$ and $a^{\tiny\textcircled{\dag}}=x^{\dag}$ by ~\cite[Lemma 4.1]{C}.
It is easy to verify that $$\begin{array}{rll}
a^{\tiny\textcircled{\dag}}a^3a^{\tiny\textcircled{\dag}}&=&x^{\dag}(x+y)^3x^{\dag}=x^{\dag}(x+y)^3x^{\dag}xx^{\dag}\\
&=&x^{\dag}(x+y)^3(x^{\dag}x)^*x^{\dag}=x^{\dag}x^3x^{\dag}\\\
&\in&R^{\dag}.
\end{array}$$
Moreover, we check that $$\begin{array}{rll}
a^{\tiny\textcircled{H}}&=&x^{\mathcal{H}}=[x^{\dag}x^3x^{\dag}]^{\dag}\\
&=&[a^{\tiny\textcircled{\dag}}a^3a^{\tiny\textcircled{\dag}}]^{\dag},
\end{array}$$ as required.

$(2)\Rightarrow (1)$ Since $a\in R^{\tiny\textcircled{\dag}}$, there exist $x,y\in R$ such that $$a=x+y, x^*y=yx=0, x\in R^{\dag}, y\in R^{nil}.$$
In this case, $a^{\tiny\textcircled{\dag}}=x^{\dag}$. Moreover, we have
$$\begin{array}{rll}
x^{\dag}x^3x^{\dag}&=&x^{\dag}(x+y)^3x^{\dag}\\
&=&a^{\tiny\textcircled{\dag}}a^3a^{\tiny\textcircled{\dag}}\\
&\in&R^{\dag}.
\end{array}$$ Hence, $x\in R^{\mathcal{H}}$. Accordingly, $a\in R^{\tiny\textcircled{H}}$.\end{proof}

\begin{cor} Let $a\in R$. Then the following are equivalent:\end{cor}
\begin{enumerate}
\item [(1)] $a\in R^{\tiny\textcircled{H}}$.
\vspace{-.5mm}
\item [(2)] The system of conditions $$x=xax, (ax)^*=ax, (xa)^*=xa, xa^3x\in R^{\dag}, a-axa\in R^{nil}$$ is consistent and it has the unique solution.
\end{enumerate}
In this case, $a^{\tiny\textcircled{H}}=(xa^3x)^{\dag}.$
\begin{proof} $(1)\Rightarrow (2)$ Set $x=a\in R^{\tiny\textcircled{H}}$. The implication is true by Theorem 3.1.

$(2)\Rightarrow (1)$ By assumption, we have $a\in R^{\tiny\textcircled{\dag}}$ and $a^{\tiny\textcircled{\dag}}a^3a^{\tiny\textcircled{\dag}}=xa^3x\in R^{\dag}$. This completed the proof by Theorem 3.1.\end{proof}

\begin{cor} Let $a,b\in R^{\tiny\textcircled{H}}$. If $ab=ba=0, a^*b=0$, then $a+b\in R^{\tiny\textcircled{H}}$. In this case, $$(a+b)^{\tiny\textcircled{H}}=a^{\tiny\textcircled{H}}+b^{\tiny\textcircled{H}}.$$\end{cor}
\begin{proof} By virtue of Theorem 3.1, we have $a,b\in R^{\tiny\textcircled{\dag}}, a^{\tiny\textcircled{\dag}}a^3a^{\tiny\textcircled{\dag}}, b^{\tiny\textcircled{\dag}}b^3b^{\tiny\textcircled{\dag}}\in R^{\dag}$ and
$$a^{\tiny\textcircled{H}}=(a^{\tiny\textcircled{\dag}}a^3a^{\tiny\textcircled{\dag}})^{\dag}, b^{\tiny\textcircled{H}}=(b^{\tiny\textcircled{\dag}}b^3b^{\tiny\textcircled{\dag}})^{\dag}.$$
Since $ab=ba=0, a^*b=b^*a=0$, by using ~\cite[Lemma 4.1]{C}, we have $a+b\in R^{\tiny\textcircled{\dag}}$ and $(a+b)^{\tiny\textcircled{\dag}}=a^{\tiny\textcircled{\dag}}+b^{\tiny\textcircled{\dag}}$.
 We directly check that $$\begin{array}{rll}
 &(a+b)^{\tiny\textcircled{\dag}}(a+b)^3(a+b)^{\tiny\textcircled{\dag}}\\
 =&(a^{\tiny\textcircled{\dag}}+b^{\tiny\textcircled{\dag}})(a^3+b^3)(a^{\tiny\textcircled{\dag}}+b^{\tiny\textcircled{\dag}})\\
 =&a^{\tiny\textcircled{\dag}}a^3a^{\tiny\textcircled{\dag}}+b^{\tiny\textcircled{\dag}}b^3b^{\tiny\textcircled{\dag}}.
 \end{array}$$ By hypothesis, we verify that
 $$\begin{array}{rll}
 \big(a^{\tiny\textcircled{\dag}}a^3a^{\tiny\textcircled{\dag}}\big)^*[b^{\tiny\textcircled{\dag}}b^3b^{\tiny\textcircled{\dag}}]&=&[aa^{\tiny\textcircled{\dag}}][a^*b][bb^{\tiny\textcircled{\dag}}]=0,\\
 \big(b^{\tiny\textcircled{\dag}}b^3b^{\tiny\textcircled{\dag}}\big)^*[a^{\tiny\textcircled{\dag}}a^3a^{\tiny\textcircled{\dag}}]&=&[bb^{\tiny\textcircled{\dag}}][b^*a][aa^{\tiny\textcircled{\dag}}]=0.
 \end{array}$$ In light of Theorem 2.1, $$\begin{array}{rll}
 (a+b)^{\tiny\textcircled{H}}&=&[a^{\tiny\textcircled{\dag}}a^3a^{\tiny\textcircled{\dag}}+b^{\tiny\textcircled{\dag}}b^3b^{\tiny\textcircled{\dag}}]^{\dag}\\
 &=&[a^{\tiny\textcircled{\dag}}a^3a^{\tiny\textcircled{\dag}}]^{\dag}+[b^{\tiny\textcircled{\dag}}b^3b^{\tiny\textcircled{\dag}}]^{\dag}\\
 &=&a^{\tiny\textcircled{H}}+b^{\tiny\textcircled{H}},
 \end{array}$$ as asserted.\end{proof}

We characterize the weak higher-order group inverse by the solvability of the following system of equations.

\begin{thm} Let $a\in R$. Then $a\in R^{\tiny\textcircled{H}}$ if and only if\end{thm}
\begin{enumerate}
\item [(1)] $a\in R^{\tiny\textcircled{\dag}};$.
\vspace{-.5mm}
\item [(2)] there exists $x\in (aa^{\tiny\textcircled{\dag}})R\bigcap R(a^{\tiny\textcircled{\dag}}a)$ such that $$xax=x, (a^2xa^2)a^{\tiny\textcircled{\dag}}=a^3a^{\tiny\textcircled{\dag}}, (a^2xa^*)^*=a^2xa^*, (a^*xa^2)^*=a^*xa^2.$$
\end{enumerate}
In this case, $a^{\tiny\textcircled{H}}=x.$
\begin{proof} $\Longrightarrow $ Let $x=(a^{\tiny\textcircled{\dag}}a^3a^{\tiny\textcircled{\dag}})^{\dag}$. Then we have $$x=
(aa^{\tiny\textcircled{\dag}})x=x(a^{\tiny\textcircled{\dag}}a)\in (aa^{\tiny\textcircled{\dag}})R\bigcap R(a^{\tiny\textcircled{\dag}}a).$$
We verify that
$$\begin{array}{rll}
xax&=&(a^{\tiny\textcircled{\dag}}a^3a^{\tiny\textcircled{\dag}})^{\dag}a(a^{\tiny\textcircled{\dag}}a^3a^{\tiny\textcircled{\dag}})^{\dag}\\
&=&(a^{\tiny\textcircled{\dag}}a^3a^{\tiny\textcircled{\dag}})^{\dag}a^{\tiny\textcircled{\dag}}a^3a^{\tiny\textcircled{\dag}}(a^{\tiny\textcircled{\dag}}a^3a^{\tiny\textcircled{\dag}})^{\dag}\\
    &=&(a^{\tiny\textcircled{\dag}}a^3a^{\tiny\textcircled{\dag}})^{\dag}=x,\\
(a^2xa^2)a^{\tiny\textcircled{\dag}}&=&a^2(a^{\tiny\textcircled{\dag}}a^3a^{\tiny\textcircled{\dag}})^{\dag}a^2a^{\tiny\textcircled{\dag}}\\
&=&a^2(a^{\tiny\textcircled{\dag}}a^3a^{\tiny\textcircled{\dag}})^{\dag}[a^{\tiny\textcircled{\dag}}a^3a^{\tiny\textcircled{\dag}}]\\
&=&a^3a^{\tiny\textcircled{\dag}}(a^{\tiny\textcircled{\dag}}a^3a^{\tiny\textcircled{\dag}})^{\dag}[a^{\tiny\textcircled{\dag}}a^3a^{\tiny\textcircled{\dag}}]\\
&=&a[a^{\tiny\textcircled{\dag}}a^3a^{\tiny\textcircled{\dag}}](a^{\tiny\textcircled{\dag}}a^3a^{\tiny\textcircled{\dag}})^{\dag}[a^{\tiny\textcircled{\dag}}a^3a^{\tiny\textcircled{\dag}}]\\
&=&a[a^{\tiny\textcircled{\dag}}a^3a^{\tiny\textcircled{\dag}}]=a^3a^{\tiny\textcircled{\dag}},\\
a^2xa^*&=&a[a^{\tiny\textcircled{\dag}}a^3a^{\tiny\textcircled{\dag}}](a^{\tiny\textcircled{\dag}}a^3a^{\tiny\textcircled{\dag}})^{\dag}a^*,\\
a^*xa^2&=&a^*(a^{\tiny\textcircled{\dag}}a^3a^{\tiny\textcircled{\dag}})^{\dag}[a^{\tiny\textcircled{\dag}}a^3a^{\tiny\textcircled{\dag}}]a,\\
(a^2xa^*)^*&=&a^2xa^*,\\
(a^*xa^2)^*&=&a^*xa^2.
\end{array}$$

$\Longleftarrow $ By hypothesis, there exists $x\in (aa^{\tiny\textcircled{\dag}})R\bigcap R(a^{\tiny\textcircled{\dag}}a)$ such that $$xax=x, (a^2xa^2)a^{\tiny\textcircled{\dag}}=a^3a^{\tiny\textcircled{\dag}}, (a^2xa^*)^*=a^2xa^*, (a^*xa^2)^*=a^*xa^2.$$
By hypothesis, $x=aa^{\tiny\textcircled{\dag}}x=xa^{\tiny\textcircled{\dag}}a$. Then we check that
$$\begin{array}{rll}
x[a^{\tiny\textcircled{\dag}}a^3a^{\tiny\textcircled{\dag}}]&=&[xa^{\tiny\textcircled{\dag}}a]a^2a^{\tiny\textcircled{\dag}}=xa^2a^{\tiny\textcircled{\dag}}\\
&=&aa^{\tiny\textcircled{\dag}}xa^2a^{\tiny\textcircled{\dag}}=(a^{\tiny\textcircled{\dag}})^*[a^*xa^2]a^{\tiny\textcircled{\dag}},\\
\big(a^{\tiny\textcircled{\dag}}a^3a^{\tiny\textcircled{\dag}}\big)x&=&\big(a^{\tiny\textcircled{\dag}}a^2\big)\big(aa^{\tiny\textcircled{\dag}}x\big)=
\big(a^{\tiny\textcircled{\dag}}a^2\big)\big(xa^{\tiny\textcircled{\dag}}a\big)\\
&=&a^{\tiny\textcircled{\dag}}[a^2xa^*](a^{\tiny\textcircled{\dag}})^*,\\
x[a^{\tiny\textcircled{\dag}}a^3a^{\tiny\textcircled{\dag}}]x&=&[xa^2a^{\tiny\textcircled{\dag}}]x=xa[aa^{\tiny\textcircled{\dag}}x]=xax=x,\\
\big(a^{\tiny\textcircled{\dag}}a^3a^{\tiny\textcircled{\dag}}\big)x\big(a^{\tiny\textcircled{\dag}}a^3a^{\tiny\textcircled{\dag}}\big)&=&
\big(a^{\tiny\textcircled{\dag}}a^3a^{\tiny\textcircled{\dag}}\big)xa^2a^{\tiny\textcircled{\dag}}\\
&=& a^{\tiny\textcircled{\dag}}[a^2xa^2]a^{\tiny\textcircled{\dag}}=a^{\tiny\textcircled{\dag}}a^3a^{\tiny\textcircled{\dag}}.
\end{array}$$ Moreover, we see that
$$\begin{array}{rll}
\big(x(a^{\tiny\textcircled{\dag}}a^3a^{\tiny\textcircled{\dag}})\big)^*&=&x(a^{\tiny\textcircled{\dag}}a^3a^{\tiny\textcircled{\dag}}),\\
\big((a^{\tiny\textcircled{\dag}}a^3a^{\tiny\textcircled{\dag}})x\big)^*&=&(a^{\tiny\textcircled{\dag}}a^3a^{\tiny\textcircled{\dag}})x,\\
\end{array}$$
Therefore $a^{\tiny\textcircled{\dag}}a^3a^{\tiny\textcircled{\dag}}\in R^{\dag}$ and $[a^{\tiny\textcircled{\dag}}a^3a^{\tiny\textcircled{\dag}}]^{\dag}=x$. This completes the proof.
\end{proof}

\begin{cor} Let $a\in R$. Then $a\in R^{\mathcal{H}}$ if and only if $a\in R^{\dag}\bigcap R^{\tiny\textcircled{H}}$.
In this case, $a^{\mathcal{H}}=a^{\tiny\textcircled{H}}.$\end{cor}
\begin{proof} $\Longrightarrow $ This is obvious.

$\Longleftrightarrow $ Since $a\in R^{\dag}$, we have $a^{\tiny\textcircled{\dag}}=a^{\dag}$. In view of Theorem 3.4,
there exists $x\in (aa^{\tiny\textcircled{\dag}})R\bigcap R(a^{\tiny\textcircled{\dag}}a)$ such that $$xax=x, (a^2xa^2)a^{\tiny\textcircled{\dag}}=a^3a^{\tiny\textcircled{\dag}}, (a^2xa^*)^*=a^2xa^*, (a^*xa^2)^*=a^*xa^2.$$ Hence,
$$a^2xa^2=[(a^2xa^2)a^{\tiny\textcircled{\dag}}]a=[a^3a^{\tiny\textcircled{\dag}}]a=a^3a^{\dag}a=a^3.$$ Therefore $a\in R^{\mathcal{H}}$. In this case, $a^{\mathcal{H}}=x=a^{\tiny\textcircled{H}}$, as desired.\end{proof}

Consider the system of equations given by

$$\hspace{20mm}  xax=x,xa=[a^{\tiny\textcircled{\dag}}a^3a^{\tiny\textcircled{\dag}}]^{\dag}a, awx=aw=[a^{\tiny\textcircled{\dag}}a^3a^{\tiny\textcircled{\dag}}]^{\dag}.  \hspace{20mm}(*)$$

\begin{lem} If the system $(*)$ of equations has a solution then it is unique.\end{lem}
\begin{proof} Assume that $x_1,x_2$ satisfy the system of equations $(*)$. Then
$$\begin{array}{c}
x_iax_i=x_i,x_ia=[a^{\tiny\textcircled{\dag}}a]^{\dag}a^{\tiny\textcircled{\dag}}a,\\
ax_i=a[a^{\tiny\textcircled{\dag}}a]^{\dag}a^{\tiny\textcircled{\dag}}
\end{array}$$ for $i=1,2$. Therefore
$$\begin{array}{rll}
x_1&=&(x_1a)x_1=[a^{\tiny\textcircled{\dag}}a^3a^{\tiny\textcircled{\dag}}]^{\dag}ax_1\\
&=&(x_2wa)(wx_1)=x_2w(awx_1)=x_2a[a^{\tiny\textcircled{\dag}}a]^{\dag}a^{\tiny\textcircled{\dag}}\\
&=&(x_2w)awx_2=x_2,
\end{array}$$ as required.\end{proof}

\begin{thm} Let $a\in R$. Then the following are equivalent:\end{thm}
\begin{enumerate}
\item [(1)] $a^{\tiny\textcircled{H}}=x$.
\vspace{-.5mm}
\item [(2)] The system of equations $$
xax=x,xa=[a^{\tiny\textcircled{\dag}}a^3a^{\tiny\textcircled{\dag}}]^{\dag}a,
ax=a[a^{\tiny\textcircled{\dag}}a^3a^{\tiny\textcircled{\dag}}]^{\dag}
$$ is consistent and it has the unique solution $x$.
\end{enumerate}
In this case, $a^{\tiny\textcircled{H}}=x.$
\begin{proof} $(1)\Rightarrow (2)$ Taking $x=[a^{\tiny\textcircled{\dag}}a]^{\dag}a^{\tiny\textcircled{\dag}}$. Then
$$\begin{array}{rll}
xax&=&(a^{\tiny\textcircled{\dag}}a^3a^{\tiny\textcircled{\dag}})^{\dag}a(a^{\tiny\textcircled{\dag}}a^3a^{\tiny\textcircled{\dag}})^{\dag}\\
&=&(a^{\tiny\textcircled{\dag}}a^3a^{\tiny\textcircled{\dag}})^{\dag}(a^{\tiny\textcircled{\dag}}a^3a^{\tiny\textcircled{\dag}})(
a^{\tiny\textcircled{\dag}}a^3a^{\tiny\textcircled{\dag}})^{\dag}\\
&=&(a^{\tiny\textcircled{\dag}}a^3a^{\tiny\textcircled{\dag}})^{\dag}\\
&=&x,\\
xa&=&(a^{\tiny\textcircled{\dag}}a^3a^{\tiny\textcircled{\dag}})^{\dag}a,\\
ax&=&a(a^{\tiny\textcircled{\dag}}a^3a^{\tiny\textcircled{\dag}})^{\dag}.
\end{array}$$

By virtue of Lemma 3.6, $x$ is the unique solution of the preceding equations, as required.

$(2)\Rightarrow (1)$ By the argument above, we have $x=[a^{\tiny\textcircled{\dag}}a]^{\dag}a^{\tiny\textcircled{\dag}}$.
Therefore $a^{\tiny\textcircled{H}}=x$, as asserted.\end{proof}

We illustrate Theorem 3.4 with the following numerical example.

\begin{exam}\end{exam} Let ${\Bbb C}^{3\times 3}$ be the ring of all $3\times 3$ complex matrices with
the involution $*$ the conjugate transpose of complex matrices. Let
$$A=\left(\begin{array}{ccc}
1  &  1+i  &  0  \\
1-i  &  2  &  0  \\
-2  &  1+i  &  0
\end{array}\right)\in {\Bbb C}^{3\times 3}.$$ Then $A=X+Y$, where $$X=\left(\begin{array}{ccc}
1  &  1+i  &  0  \\
1-i  &  2  &  0  \\
0  &  0  &  0
\end{array}\right), Y=\left(\begin{array}{ccc}
0&0&0\\
0&0&0\\
-2&1+i&0\end{array}\right).$$

Let $Z=\frac{1}{9}
 \left(\begin{array}{ccc}
 1  &  1+i  &  0  \\
1-i  &  2  &  0  \\
0  &  0  &  0
\end{array}\right)$. By the direct calculation, we check that
$$ZXZ=Z,  (XZ)^{\ast}=XZ,  (ZX)^{\ast}=ZX,  XZX=X.$$
So $X^{\dag}=Z$.  \\

Furthermore, $X$ and $Z$ satisfy that
$$ZXZ=Z, X^2ZX^2=X^3, (X^2ZX^{\ast})^{\ast}=X^2ZX^{\ast}, (X^{\ast}ZX^2)^{\ast}=X^{\ast}ZX^2.$$
This implies that $X$ have higher order group inverse and $X^{\mathcal{H}}=Z.$\\

We further verify that $Y$ is nilpotent and $X^*Y=YX=0$. According to Theorem 3.4, $A$ has weak higher order group inverse and
$$A^{\tiny\textcircled{H}}=X^{\mathcal{H}}=\frac{1}{9}
 \left(\begin{array}{ccc}
 1  &  1+i  &  0  \\
1-i  &  2  &  0  \\
0  &  0  &  0
\end{array}\right).$$

\end{document}